\RequirePackage{fix-cm}
\documentclass[12pt]{article}

\usepackage{amsthm}
\usepackage{amsmath}
\usepackage{amsfonts}
\usepackage{graphicx}
\usepackage{algorithm2e}
\newtheorem{definition}{Definition}
\newtheorem{theorem}{Theorem}
\newtheorem{conjecture}{Conjecture}
\newtheorem{lemma}{Lemma}
\newtheorem{corollary}{Corrollary}
\begin{document}

\title{ Violation Probability of Goldbach Conjecture and Generalization to Prime-like Distributions}
\title{  Goldbach Conjecture: Violation Probability and Generalization to Prime-like Distributions}
\markright{Generalizing the Goldbach Conjecture}
\author{  Ameneh Farhadian}




\date{10/2/2025}
\maketitle

\begin{abstract}
Due to the distribution of primes among integers, we establish an upper bound for the probability $\mathbb{P}_n$ that the Goldbach conjecture fails.  Assuming the conjecture holds true for all even number less than $2N$, we prove this probability is less than $e^{-N^\alpha}$, where $
\alpha = 1 - \frac{2\ln\ln N}{\ln N}$.
 For large $N$, this probability becomes vanishingly small, effectively precluding the existence of counterexamples in practice. If $N =4 \times 10^{18}$, the probability of a counterexample is less than $e^{-10^{15}}$. Our approach fundamentally depends on the distributional properties of primes rather than their primality per se. This perspective enables a natural generalization of the conjecture to non-prime subsets of integers that exhibit similar distributional characteristics. As a concrete example, we construct new subsets by applying random $\pm 1$ shifts to primes, which preserve the essential prime-like distributional properties. Computational verification confirms that this generalized Goldbach conjecture holds for all even integers up to $2 \times 10^{8}$ within these modified subsets.

\end{abstract}

\noindent

\section{Introduction}
The well-known Goldbach conjecture states that every even integer greater than 2 can be expressed as the sum of two prime numbers. Proposed in 1724, this conjecture remains unproven despite significant efforts  \cite{wang2002goldbach}. However, computer calculations have shown that the conjecture holds for all integers less than $4 \times 10^{18} $ \cite{oliveira2014empirical}.
%

Theodor Estermann \cite{estermann1938goldbach} proved in 1938 that almost all even numbers can be expressed as the sum of two primes, in the sense that the proportion of even numbers $\leq N$ admitting such a representation approaches 1 as $N \to \infty$. Earlier, Hardy and Littlewood \cite{hardy1923some} established that the count of even numbers $\leq X$ violating Goldbach's conjecture is bounded above by $X^{1/2 + c}$ for some small $c > 0$. A stronger result was obtained by Montgomery and Vaughan \cite{montgomery1975exceptional}, who demonstrated that almost all even numbers satisfy the conjecture. Specifically, they proved the existence of constants $c, C > 0$ such that for all sufficiently large $N$, at most $CN^{1-c}$ even integers below $N$ cannot be written as the sum of two primes. This implies that the set of even numbers failing Goldbach's conjecture has asymptotic density zero.

Here, we adopt a slightly different approach to the problem, a modified and extended version of the method originally proposed by Clarke and Shannon\cite{clarke198367}. They applied a combinatorial approach and demonstrated that it is highly unlikely that Goldbach's conjecture is false. Here, for an even number \( 2n \), we form two sets, \( A_n \) and \( B_n \):
\begin{itemize}
    \item  Set $A_n$: Contains the differences \( n - p \) for all primes \( p \) less than \( n \).
    \item Set $B_n$: Contains the differences \( q - n \) for all primes \( q \) between \( n \) and \( 2n \).
\end{itemize}
The conjecture holds if the two sets \( A_n \) and \( B_n \) intersect (i.e., if \( A_n \cap B_n \neq \emptyset \)). Based on the distribution of primes, we compute the probability of the conjecture failing for a given even number---that is, the probability that \( A_n \cap B_n = \emptyset \). This probability is then summed over all even integers from \( 2N \) to infinity, where \( 2N \) is the largest even number for which the conjecture has been numerically verified.

Since our approach relies solely on the distribution of primes rather than primality testing, we extend our findings to all subsets of integers that follow a distribution similar to that of the primes.

In \textbf{Section 2}, we formally define distributions resembling the prime numbers and present a generalized version of the Goldbach conjecture for such subsets. \textbf{Section 3} derives the probability of the conjecture being violated for a single even number and sums this probability over all even integers from \( 2N \) to infinity. We show that this cumulative probability is bounded above by \( e^{-N^{\alpha}} \), where
\[
\alpha \approx 1 - \frac{2\ln{\ln{N}}}{\ln{N}}.
\]
To provide further intuition, \textbf{Section \ref{comp}} computes this probability explicitly for specific values of \( N \). The paper concludes with a summary of results in the final section.
\section{Distribution similar to prime numbers}\label{not_nec}
Here, we define the subsets of natural numbers  whose distribution is similar to prime numbers. Then, we compute the probability of violating the conjecture, not only for primes, but also for any subset whose distribution among natural numbers is similar to primes.

We know that the number of prime numbers less than or equal to some integer $n$ is denoted by $\pi(n)$. It is proved that for large $n$, $\pi(n)$ is approximated by $ \frac{n}{\ln(n)}  $\cite{ireland1990classical}. Now, we define distribution similar to prime numbers\cite{farhadian2022simple}.
\begin{definition}
Let $Q$ be a subset of natural numbers and $\pi_Q(n)$ be the number of integers less than or equal $n$ belong to $Q$. We say  that \textit{the distribution of $Q$ in natural numbers is similar to prime numbers}, if there exists $c \in \mathbb{N}$ such that
\begin{equation}
\vert \pi_Q(n) -\pi(n) \vert <c
\end{equation}
 for all integer $n$.
\end{definition}

According to the definition, if $Q$ is a subset of natural numbers whose distribution is similar to prime numbers, then \[ \pi(n)-c<\pi_Q(n)<\pi(n)+c\]It means that the number of $Q$ elements less than $n$ is equal to the number of primes less then $n$ by at most $c$ error. Therefore, for large $n$, which $\pi(n)$ is approximated by $\frac{n}{\ln(n)}$, we have
 \begin{equation}
 \pi_Q(n)\approx \frac{n}{\ln(n)}
\end{equation}

By randomly addition of +1 and -1 to prime numbers, we can construct a subset $Q'$ of natural numbers whose distribution is similar to prime numbers. For such subsets, we have  $\vert \pi_{Q'}(n) -\pi(n) \vert <2$. In fact, $P_t$ is shifted prime numbers by $t$.\\

\begin{definition}
Let $Q$ be a subset of natural numbers. We say \textit{the Goldbach conjecture holds true for $Q$}, if there exists $N_0  \in \mathbb{N}$ such that for every even integer $2n$ greater than $N_0$, there exist $q_1, q_2 \in Q$ such that $2n=q_1+q_2$.
\end{definition}
Now, we introduce a subset of integers whose distribution is similar to primes. Then we prove that the Goldbach conjecture holds true for this subset provided that it is true for prime numbers.
Let $P$ be the set of prime numbers and $t$ be an integer number. We define
\begin{equation}
P_t= \lbrace x+t | x \in P \rbrace
\end{equation}
We have $\pi_{P_t}(n)=\pi(n-t)$. Thus,
\begin{equation}
\vert \pi_{P_t}(n) -\pi(n) \vert =\vert \pi(n-t) -\pi(n) \vert<t+1
\end{equation} Thus, the distribution of $P_t$ in natural numbers is similar to prime numbers.
\begin{theorem}\label{asl}
Let $P$ be the set of prime numbers and $c$ be an integer number. We define $P_c= \lbrace x+c | x \in P \rbrace $. If the Goldbach conjecture holds true for the prime numbers, then the Goldbach conjecture also holds true for $P_c$. (That is for any even number $n$ greater that $2c+2$, there exist $p_c,q_c \in P_c$ such that $2n=p_c+q_c$.)
\end{theorem}
\begin{proof}
Assume the Goldbach conjecture holds true for $n>2$. We want to show that there exist $p_c, q_c \in P_c$ such that $p_c+q_c=2n$ for every even integer $2n>2c+2$.  For $2n>2c+2$, we have $2n-2c>2$. Thus, due to the Goldbach conjecture, there exist prime numbers $p$ and $q$ such that $p+q=2n-2c$. Consequently, $p+c ,q+c \in P_t$ and \[ (p+c)+(q+c)=(p+q)+2c=2n.\] Therefore, the Goldbach conjecture also holds true for $P_c$.
\end{proof}
Altough the Theorem \ref{asl}  is very simple, but it has an important result. The Goldbach conjecture holds true for a set whose elements are not prime, but with the same distribution that prime numbers have. It stimulate us to ask whether "being prime" is a necessarily key concept in the Goldbach conjecture?


\begin{conjecture}
(Generalization of Goldbach Conjecture)
Let $Q$ be a subset of natural numbers whose distribution is similar to primes. There exists $N_Q \in \mathbb{N}$, such that for any even integer $2n$ greater than $N_Q$, there exist $q_1, q_2 \in Q$ such that $q_1+q_2=2n$.
\end{conjecture}
Furthermore, we verified this generalization empirically. We constructed several new subsets with distributions resembling the primes by randomly adding \(\pm 1\) to the prime numbers. We then tested the Goldbach conjecture for these modified sets up to \(2n \leq 2 \times 10^8\) using computational methods. As expected, the Goldbach conjecture holds for these reconstructed sets when \(2n > 40\).

\section{The Probability of violation of Conjecture}
In the previus section, the generalization of the Goldbach conjecture is suggested and it is  verified practically for different subsets whose distribution are similar to primes. We saw that the Generalized Conjecture holds true for \(2n \leq 2 \times 10^8\).

Suppose there exists a prime $p < n$ and a prime $q > n$ such that their distances to $n$ are equal, that is,
$n - p = q - n$.
This immediately implies that
$
2n = p + q.
$
We can therefore interpret the satisfaction of Goldbach's condition for an even number $2n$ as a probabilistic event dependent on the distribution of primes.  Consequently, we may compute the probability of the conjecture holding (or failing) for $2n$ purely in terms of the distribution of primes among integers. Crucially, this approach allows us to estimate the likelihood of the conjecture being satisfied independently of specific primality tests, relying instead solely on the distributional properties of primes.

In this section, we compute the probability of violating the generalized Goldbach conjecture under the assumption that the conjecture holds for all even numbers below $2N$. We demonstrate that this violation probability becomes vanishingly small for sufficiently large $N$. Specifically, we prove that if the Goldbach conjecture is valid for all even integers less than $2N$, then the probability of a counterexample existing is bounded above by $e^{-\sqrt{N}}$.

Given that the Goldbach conjecture has been numerically verified up to $4 \times 10^{18}$ \cite{oliveira2014empirical}, our results imply an extraordinarily small upper bound for the existence of counterexamples - less than $e^{-\sqrt{2 \times 10^{18}}}$. This bound is so exceptionally stringent that it effectively precludes the possibility of finding any counterexample in practice.
%
 More intriguingly, the violation probability diminishes exponentially with increasing $N$, making counterexamples practically non-existent at astronomical scales.

\begin{theorem} \label{asl_prob}
Let $Q$ be  a prime-like distribution subset. The probability that generalized Goldbach conjecture violates for $Q$ is less than
\begin{equation}
 e^{-\sqrt{N}}
\end{equation}
provided that the conjecture holds true for $Q$ for even numbers up to $2N$. (gratear than $N_0$)
\end{theorem}
Theorem~\ref{asl_prob} has considerable consequences. Since the Goldbach conjecture has been verified
for all even numbers less than $2 \times 10^{18}$ in practice, we conclude that the probability that the
conjecture does not hold is less than $e^{-10^9}$. This probability is extremely low, effectively
zero for practical purposes.
In addition, we have shown that the generalized conjecture holds true for even numbers less than $10^8$. Therefore, the probability of violation is less than $e^{-10^4}(\approx 10^{-4343})$.

To prove Theorem \ref{asl_prob}, first we present and prove the following lemma.
\begin{lemma}\label{f}
Let $Q$ be a subset of natural numbers whose distribution  is similar to prime numbers. The probability that  $p, q $ does not exist in $Q$ such that $2n=p+q$ is less than
\begin{equation}  \mathbb{P}(n)= \exp{(-n/\ln^2{n})}
\end{equation}
\end{lemma}

\begin{proof}
if there exist $q_1, q_2 \in Q$, we can rewrite it in the form
 $$n-q_1=q_2-n$$
  Due to this form of writing, we can say $n$ is exactly  allocated in the middle of $q_1$ and $q_2$. Thus we account the distances of primes less than to $n$ and also the distances of  $n$ to primes which are more than $n$. If we can find an equal value in these two subset, we have found two primes that $n$ is placed exactly middle of them.

Let $A_n$ be the set of  distances of $n$  to the elements of $Q$ which are less than or equal to $n$. Let $B_n$ be the distances of $n$ to the elements of $Q$ between $n$ and $2n$, respectively. That is,
\begin{equation}
A_n := \lbrace n-q \mid \text{$ q\leq n$ and $ q \in Q$}  \rbrace
\end{equation}
\begin{equation}
B_n := \lbrace q-n \mid    \text{$ n \leq q < 2n$ and $ q \in Q$} \rbrace
\end{equation}
Clearly, we have  $A_n , B_n \subset \{ 0,1, \dots, n-1\}$.  The Goldbach conjecture does not hold true for even integer $2n$ if and only if  we have $A_n \cap B_n = \emptyset$.

 Let  $\vert A_n \vert= k_1$ and  $\vert B_n \vert= k_2$. Since  $A_n , B_n \subset \{0, 1, \dots, n-1\}$, the probability that $A_n \cap B_n = \emptyset $ is
\begin{equation}
\mathbb{P}(n)=  \frac {       {n \choose{k_1}}         {n-k_1 \choose{k_2}}        } { {{n}\choose{k_1}}   {{n}\choose{k_2}}}
\end{equation}

This formula is similar to \cite{clarke198367} where the possible sum on primes are considered.  Here, we have defined subsets $A_n$ and $B_n$ which provides more intuition. 
We have $ \pi_Q(n) \approx \frac {n} {\ln{n}}$ for large $n$. Therefore, there are $\frac {n} {\ln{n}}$ number of $Q$ elements between 1 and $n$ and $\frac{2n} {\ln{2n}}- \frac {n}{\ln{n}} \approx \frac{n}{\ln{n}}$  $Q$ elements between $n$ and $2n$.
Substituting $k_1=k_2=\frac {n} {\ln{n}}$ in the above equation, we have
\begin{equation}
 \mathbb{P}(n)< (\frac{n-k_1} {n})^{k_2}=(\frac{n-n/\ln{n}} {n})^{n/\ln{n}}=(1-\frac {1}{\ln {n}})^{n/\ln{n}}
\end{equation}
For large $n$, we have
\begin{equation} \mathbb{P}(n)< \exp{(-n/\ln^2{n})}
\end{equation}  $\square$
\end{proof}

According to Lemma \ref{f}, the probability of violating the Goldbach conjecture for large even number $2n$  for any subset of natural numbers whose distribution is similar to prime numbers is less than  $  \mathbb{P}(n)=\exp{(-n/\ln^2{n})}$.
Here, we have the probability that $n$ does not hold the conjecture.
Now, we are ready to prove Theorem \ref{asl_prob}.

\begin{proof}(Theorem \ref{asl_prob})\\
According to Lemma \ref{f}, the probability of violating the Goldbach conjecture for large even number $2n$  for any subset of natural numbers whose distribution is similar to prime numbers is less than  $  \mathbb{P}(n)=\exp{(-n/\ln^2{n})}$.
 To compute the probability of violating the conjecture, we should sum the probability for all possible even numbers more than $2N$, provided that the conjecture is verified for even numbers less than $2N$. That is, the probability of violating the conjecture is
\begin{equation}
 \mathbb{P}=\sum_{n=N}^{\infty} \mathbb{P}(n)< \sum_{n=N}^{\infty} \exp{(-n/\ln^2{n})}
\end{equation}
To compute this summation, we compute the following integral
\begin{equation}
 \sum_{n=N}^{\infty} \exp{(-n/\ln^2{n})} \approx  \int_{n=N}^{\infty} \exp{(-n/\ln^2{n})}dn
\end{equation}
Since $N$ is a sufficiently large number, we have
\begin{equation}
\frac{n}{\ln^2{n}}>\sqrt{n}+\ln{2\sqrt{n}}
\end{equation}
Therefore,

\begin{equation}
\exp{(-\frac{n}{\ln^2{n}})}<e^{-(\sqrt{n}+\ln{2\sqrt{n}})}=\frac {e^{-\sqrt{n}}}{{2\sqrt{n}}}
\end{equation}
Then,
\begin{equation}
\int_{N}^{\infty}e^{-\frac{n}{\ln^2{n}}}dn<\int_{N}^{\infty}\frac {e^{-\sqrt{n}}}{{2\sqrt{n}}}dn=e^{-\sqrt{N}}
\end{equation}
\end{proof}
The result of the previous theorem is improved in the next theorem. Instead of $e^{-\sqrt{N}}$, $e^{-N^\alpha}$ is suggested where $1/2<\alpha<1$.
\begin{theorem}
Assuming the Goldbach conjecture holds true for all even numbers less than $2N$, the probability that Goldbach conjecture does not hold for $Q$ is less than
$ e^{-N^\alpha}$ where $\alpha= 1- \frac{2\ln{\ln{N}}}{\ln{N}}$.
\end{theorem}
\begin{proof}
We know that  $n^{\frac{1}{2}}<<\frac{n}{\ln^2{n}}$ for large $n$. We can improve $n^{\frac{1}{2}}$ to $n^\alpha$ where $\frac{1}{2}<\alpha<1$. Suppose  $\frac{1}{2}<\alpha<1$ is such that
\begin{equation} \label{alpha}
 \frac{N}{\ln^2{N}} = N^\alpha +(\alpha-1)\ln{N}  +\ln{\alpha}
\end{equation}

Since the growth rate of  left side of the above equation is significantly more than righthand, we have
\begin{equation}
\frac{n}{\ln^2{n}} >n^\alpha +(\alpha-1)\ln{n}  +\ln{\alpha}    \text{      \space{  }   for \( n> N \)}
\end{equation}
Thus,
\begin{equation}
e^{-\frac{n}{\ln^2{n}}} <{\alpha}  e^{-n^\alpha} n^{\alpha-1}
\end{equation}
Consequently, we have
\begin{equation}
\int_{N}^{\infty} e^{-\frac{n}{\ln^2{n}}} dn < \int_{N}^{\infty} \alpha e^{-n^\alpha} n^{\alpha-1} dn =e^{-N^\alpha}
\end{equation}
In Equation \ref{alpha}, value of $(\alpha-1)\ln{N}  +\ln{\alpha}$ is negligible to $N^\alpha$ for large $N$, thus the value of $\alpha$ can be estimated by
\begin{equation}
\alpha \approx 1 - \frac{2\ln{\ln{N}}}{\ln{N}}
\end{equation}
%

\end{proof}
\begin{corollary} Since the Goldbach Conjecture is verified for even number no mare than $4 \times 10^{18}$ \cite{oliveira2014empirical}, the failure probability of conjecture is less than $e^{-10^{15}}$.
\end{corollary}
\begin{proof}
Since $N=4 \times 10^{18}$ \cite{oliveira2014empirical}, we have $\alpha = 1- \frac{2\ln{\ln{4 \times 10^{18}}}}{\ln{4 \times 10^{18}}} \approx .8$.
\end{proof}
The value of the above probability is extremely low. It means that  the probability of finding counterexample in larger number is zero in practice.

\section{Computational results}\label{comp}

Let us study the behavior of the function \(\mathbb{P}(n)\) introduced in Lemma \ref{f}. The function
\[
\mathbb{P}(n) = \exp\left(-\frac{n}{\ln^2 n}\right)
\]
is a damping function. We have
\[
\mathbb{P}(10000) < 10^{-51} \quad \text{and} \quad \mathbb{P}(40000) < 10^{-154}.
\]
As \(n\) grows, the probability that the conjecture does not hold for any subset whose distribution is similar to prime numbers tends to zero.

Moreover, the summation of \(\mathbb{P}(n)\) from \(N\) to infinity is negligible for large \(N\). For instance, we have
\begin{equation}
    \sum_{n=20000}^\infty \mathbb{P}(n) \approx 10^{-86}
    \label{eq:sum1}
\end{equation}
and
\begin{equation}
    \sum_{n=50000}^\infty \mathbb{P}(n) \approx 10^{-183}.
    \label{eq:sum2}
\end{equation}

This implies that finding a counterexample to the conjecture for a subset of natural numbers with a distribution similar to the primes is highly improbable. Therefore, the Goldbach conjecture, originally proposed for the prime numbers, can be generalized to subsets of natural numbers with a prime-like distribution.

\section{Conclusion}

Considering the distribution of primes among integers, we computed the probability $\mathbb{P}_n$ that there exist no primes $p$ and $q$ such that $2n = p + q$. Assuming the conjecture holds for even numbers less than $2N$, we then computed $\sum_{n=N}^{\infty} \mathbb{P}_n$ as the probability of violating the conjecture.

We found that the probability of violating the Goldbach Conjecture by this method is less than $ e^{-N^{\alpha}}  $ where 
$\alpha = 1 - \frac{2\ln{\ln{N}}}{\ln{N}}.$
Since the conjecture has been verified numerically for even numbers up to $N = 4 \times 10^{18}$ \cite{oliveira2014empirical}, the probability is less than $e^{-10^{15}}$. This value is extremely low and presents no counterexample in practice.

Moreover, since the computed probability is based on the distribution of primes rather than primality, we could generalize the Goldbach Conjecture to non-prime integers. It is sufficient to have a subset of integers whose distribution among integers resembles that of primes. We constructed such a subset by randomly adding $\pm 1$ to primes and then verified the conjecture for even numbers up to $2 \times 10^8$. As expected, the conjecture remains true for this new reconstructed subset.

\section*{Data Availability Statement:} Data sharing is not applicable to this article as no new data other than
 given in the paper were created or analyzed in this study.
\section*{Conflict of Interest Statement:} The authors declare that there is no conflict of interest.

\begin{thebibliography}{1}

\bibitem{clarke198367}
JH~Clarke and AG~Shannon.
\newblock 67.2 a combinatorial approach to goldbach conjecture.
\newblock {\em The Mathematical Gazette}, 67(439):44--46, 1983.

\bibitem{estermann1938goldbach}
Theodor Estermann.
\newblock On goldbach's problem: Proof that almost all even positive integers
  are sums of two primes.
\newblock {\em Proceedings of the London Mathematical Society}, 2(1):307--314,
  1938.

\bibitem{farhadian2022simple}
Ameneh Farhadian and Hamid~Reza Fanai.
\newblock A simple explanation for the goldbach conjecture.
\newblock {\em arXiv preprint arXiv:2211.02865}, 2022.

\bibitem{hardy1923some}
Godfrey~H Hardy and John~E Littlewood.
\newblock Some problems of ‘partitio numerorum’; iii: On the expression of
  a number as a sum of primes.
\newblock {\em Acta mathematica}, 44(1):1--70, 1923.

\bibitem{ireland1990classical}
Kenneth Ireland, Michael~Ira Rosen, and Michael Rosen.
\newblock {\em A classical introduction to modern number theory}, volume~84.
\newblock Springer Science \& Business Media, 1990.

\bibitem{montgomery1975exceptional}
H~Montgomery and R~Vaughan.
\newblock The exceptional set of goldbach's problem.
\newblock {\em Acta arithmetica}, 27:353--370, 1975.

\bibitem{oliveira2014empirical}
Tom{\'a}s Oliveira~e Silva, Siegfried Herzog, and Silvio Pardi.
\newblock Empirical verification of the even goldbach conjecture and
  computation of prime gaps up to 4 $\times 10^{18}$.
\newblock {\em Mathematics of Computation}, 83(288):2033--2060, 2014.

\bibitem{wang2002goldbach}
Yuan Wang.
\newblock {\em The Goldbach Conjecture}, volume~4.
\newblock World scientific, 2002.

\end{thebibliography}

%
%

\end{document}